\documentclass[11pt]{amsart}
\usepackage{latexsym}
\usepackage{amssymb,amsmath,layout}

\theoremstyle{plain}
\newtheorem{thrm}{Theorem}[section]

\newtheorem{rmrk}[thrm]{Remark}

\begin{document}

\newcommand{\SL}{\mathcal L^{1,p}( D)}
\newcommand{\Lp}{L^p( Dega)}
\newcommand{\CO}{C^\infty_0( \Omega)}
\newcommand{\Rn}{\mathbb R^n}
\newcommand{\Rm}{\mathbb R^m}
\newcommand{\R}{\mathbb R}
\newcommand{\Om}{\Omega}
\newcommand{\Hn}{\mathbb H^n}
\newcommand{\aB}{\alpha B}
\newcommand{\eps}{\ve}
\newcommand{\BVX}{BV_X(\Omega)}
\newcommand{\p}{\partial}
\newcommand{\IO}{\int_\Omega}
\newcommand{\bG}{\boldsymbol{G}}
\newcommand{\bg}{\mathfrak g}
\newcommand{\bz}{\mathfrak z}
\newcommand{\bv}{\mathfrak v}
\newcommand{\Bux}{\mbox{Box}}
\newcommand{\e}{\ve}
\newcommand{\X}{\mathcal X}
\newcommand{\Y}{\mathcal Y}
\newcommand{\W}{\mathcal W}
\newcommand{\la}{\lambda}
\newcommand{\vf}{\varphi}
\newcommand{\rhh}{|\nabla_H \rho|}
\newcommand{\Ba}{\mathcal{B}_\beta}
\newcommand{\Za}{Z_\beta}
\newcommand{\ra}{\rho_\beta}
\newcommand{\na}{\nabla_\beta}
\newcommand{\vt}{\vartheta}

\numberwithin{equation}{section}

\newcommand{\RN} {\mathbb{R}^N}
\newcommand{\Sob}{S^{1,p}(\Omega)}
\newcommand{\Dxk}{\frac{\partial}{\partial x_k}}
\newcommand{\Co}{C^\infty_0(\Omega)}
\newcommand{\Je}{J_\ve}
\newcommand{\beq}{\begin{equation}}
\newcommand{\bea}[1]{\begin{array}{#1} }
\newcommand{\eeq}{ \end{equation}}
\newcommand{\ea}{ \end{array}}
\newcommand{\eh}{\ve h}
\newcommand{\Dxi}{\frac{\partial}{\partial x_{i}}}
\newcommand{\Dyi}{\frac{\partial}{\partial y_{i}}}
\newcommand{\Dt}{\frac{\partial}{\partial t}}
\newcommand{\aBa}{(\alpha+1)B}
\newcommand{\GF}{\psi^{1+\frac{1}{2\alpha}}}
\newcommand{\GS}{\psi^{\frac12}}
\newcommand{\HFF}{\frac{\psi}{\rho}}
\newcommand{\HSS}{\frac{\psi}{\rho}}
\newcommand{\HFS}{\rho\psi^{\frac12-\frac{1}{2\alpha}}}
\newcommand{\HSF}{\frac{\psi^{\frac32+\frac{1}{2\alpha}}}{\rho}}
\newcommand{\AF}{\rho}
\newcommand{\AR}{\rho{\psi}^{\frac{1}{2}+\frac{1}{2\alpha}}}
\newcommand{\PF}{\alpha\frac{\psi}{|x|}}
\newcommand{\PS}{\alpha\frac{\psi}{\rho}}
\newcommand{\ds}{\displaystyle}
\newcommand{\Zt}{{\mathcal Z}^{t}}
\newcommand{\XPSI}{2\alpha\psi \begin{pmatrix} \frac{x}{|x|^2}\\ 0 \end{pmatrix} - 2\alpha\frac{{\psi}^2}{\rho^2}\begin{pmatrix} x \\ (\alpha +1)|x|^{-\alpha}y \end{pmatrix}}
\newcommand{\Z}{ \begin{pmatrix} x \\ (\alpha + 1)|x|^{-\alpha}y \end{pmatrix} }
\newcommand{\ZZ}{ \begin{pmatrix} xx^{t} & (\alpha + 1)|x|^{-\alpha}x y^{t}\\
     (\alpha + 1)|x|^{-\alpha}x^{t} y &   (\alpha + 1)^2  |x|^{-2\alpha}yy^{t}\end{pmatrix}}
\newcommand{\norm}[1]{\lVert#1 \rVert}
\newcommand{\ve}{\varepsilon}

\title[Vanishing order etc. ]{ Sharp Vanishing order of solutions to  Stationary Schr\"odinger equations on Carnot groups of arbitrary step}

\author{Agnid Banerjee}
\address{TIFR-CAM, Bangalore- 560065
} \email[Agnid Banerjee]{agnidban@gmail.com}

%
%
%
\begin{abstract}
	Based on a variant of the frequency function approach of  Almgren(\cite{Al}), under appropriate assumptions  we  establish  an optimal upper bound on the  vanishing order of solutions to stationary Schr\"odinger equations associated to sub-Laplacian on  Carnot groups of arbitrary step. Such a bound provides a quantitative form of strong unique continuation and can be thought of as a subelliptic analogue of the recent results obtained by  Bakri (\cite{Bk}) and Zhu (\cite{Zhu1}) for the standard Laplacian.
\end{abstract}

\maketitle

\section{Introduction}\label{S:intro}
We say that the vanishing order of a function $u$ is $\ell$ at $x_0$, if $\ell$ is the largest  integer such that $D^{\alpha} u=0$ for all $|\alpha| \leq \ell$, where $\alpha$ is a multi-index. In the papers \cite{DF1}, \cite{DF2}, Donnelly and Fefferman showed that if $u$ is an eigenfunction with eigenvalue  $\lambda$ on a  smooth, compact and connected $n$-dimensional Riemannian manifold $M$, then  the maximal vanishing order of $u$ is less than $C \sqrt{\lambda}$ where $C$ only depends on the manifold $M$. Using this estimate, they showed  that $H^{n-1}(x: u_{\lambda}(x)=0)\leq C  \sqrt{\lambda}$ where $u_{\lambda}$ is the eigenfunction corresponding to $\lambda$ and therefore  gave a complete  answer to a famous  conjecture of Yau ( \cite{Yau}). We note that the zero set of $u_{\lambda}$ is referred to as the nodal set. This order of vanishing is sharp. If, in fact, we consider $M = \mathbb S^n \subset \R^{n+1}$, and we take the spherical harmonic $Y_\kappa$ given by the restriction to $\mathbb S^n$ of the function $f(x_1,...,x_n,x_{n+1}) = \Re (x_1 + i x_2)^\kappa$, then one has $\Delta_{\mathbb S^n} Y_\kappa = - \lambda_\kappa Y_\kappa$, with $\lambda_\kappa = \kappa(\kappa+n-2)$, and the order of vanishing of $Y_\kappa$ at the North pole $(0,...,0,1)$ is precisely $\kappa = C \sqrt {\lambda_\kappa}$. 

In his work  \cite{Ku}  Kukavica considered the more general problem
\begin{equation}\label{e1}
\Delta u = V(x) u,
\end{equation}
where $V\in W^{1, \infty}$, and showed that the maximal vanishing order  of $u$ is bounded above by $C( 1+ ||V||_{W^{1, \infty}})$. He also conjectured that the rate of vanishing order of $u$ is less than or equal to $C(1+ ||V||_{L^{\infty}}^{1/2})$, which agrees with the Donnelly-Fefferman result when $V = - \lambda$. Employing Carleman estimates,  Kenig in \cite{K} showed that the rate of  vanishing order  of $u$  is less than $C(1+ ||V||_{L^{\infty}}^{2/3})$, and that furthermore the exponent $\frac{2}{3}$ is sharp for complex  potentials $V$ based on a counterexample of Meshov. (see \cite{Me}).

Recently, the rate of vanishing order of $u$ has been shown to be less than $C(1+ ||V||_{W^{1, \infty}}^{1/2})$  independently by Bakri in \cite{Bk} and Zhu in \cite{Zhu1}. Bakri's approach is based on  an extension of the Carleman method in \cite{DF1}. In this connection, we also quote the recent interesting paper by R\"uland \cite{Ru}, where Carleman estimates are used to obtain related quantitative unique continuation results for nonlocal Schr\"odinger operators such as $(-\Delta)^{s/2} + V$. On the other hand, Zhu's approach is based on a variant of the frequency function approach  employed by Garofalo and Lin in \cite{GL1}, \cite{GL2}), in the context of strong  unique  continuation problems. Such variant consists in studying the growth properties of the following average of the Almgren's height function 
\[
H(r) = \int_{B_r(x_0)} u^2 (r^2 - |x-x_0|^2)^\alpha dx, \ \ \ \ \ \ \ \ \alpha> - 1,
\]
first introduced by Kukavica in \cite{Ku2} to study quantitative unique continuation and vortex degree estimates for solutions of the Ginzburg-Landau equation. 

In \cite{Bk} and  \cite{Zhu1} it was assumed that $u$ be a solution in $B_{10}$ to 
\begin{equation}
\Delta u= Vu,
\end{equation}
with $||V||_{W^{1, \infty}} \leq M$ and  $||u||_{L^{\infty}} \leq C_0$, and that furthermore $\sup_{B_1} |u| \geq 1$. Then, it was proved that $u$ satisfies the sharp growth estimate
\begin{equation}\label{gn}
||u||_{L^{\infty}(B_r)} \geq B r^{C(1 + \sqrt{M})},
\end{equation}
where $B, C$ depend only on $n$ and $C_0$.  Such an estimate  has been recently   extended to stationary Schr\"odinger equations associated to   generalized  Baouendi Grushin operators  in \cite{BG1} and also    for elliptic equations with Lipschitz  principle part at the boundary of Dini domains in \cite{BG2}.  Over here, we would like to refer to\cite{Ba}, \cite{Gr1} and \cite{Gr2} for a detailed account on Baouendi-Grushin operators and  corresponding hypoellipticity results. 

\medskip

Therefore given the current  interest in quantitative forms of strong unique  continuation and the crucial role played by them in the past  to get Hausdorff measure estimates on the nodal sets   as in \cite{DF1} and \cite{DF2} has provided  us with a natural motivation to study  quantitative uniqueness  for elliptic equations on Carnot groups. More precisely, we analyze equations of the form
\begin{equation}\label{e0}
\Delta_{H} u= V u,
\end{equation}
where  $\Delta_{H}$ is the sub-Laplacian on  a Carnot group $\mathbb{G}$ ( see \eqref{Def} below ) and the discrepancy $E_u$ of the solution $u$( see \eqref{dis} for the definition) at the identity $e$   satisfies the growth assumption \eqref{dis1}. The growth assumption \eqref{dis1} can  be thought of as the measure of a  certain symmetry type property of $u$ and we have  kept   a brief  discussion on this aspect  in Section 2. 

\medskip

 The assumptions on the potential function $V$ are specified in \eqref{vasump} in the next section. They represent the counterpart  on  $\mathbb{G}$ with respect to certain non-isotropic dilations  of the following Euclidean requirements
 \begin{equation}\label{Vu}
 |V(x)|\le M,\ \ \ \ \ \ \ |<x,DV(x)>| \le M,
 \end{equation}
for the classical Schr\"odinger equation $\Delta u = V u$ in $\Rn$.  Such non-isotropic dilations will be described in Section 2. 

 \medskip

Now in the case of Carnot groups, unlike the Euclidean case, the reader  should  notice  that although   we have an additional  assumption  \eqref{dis1} on the discrepancy  $E_{u}$ of $u$, it is still  not very restrictive in the sense that  strong unique continuation property is in general not true for solutions to \eqref{e0}. This follows from some interesting work of Bahouri (\cite{Bah}) where the author showed that unique continuation is not true for  even smooth and  compactly supported perturbations of the sub-Laplacian. Therefore, one cannot expect any quantitative estimates to hold either without further assumptions. Once we introduce the appropriate notion in Section 2, the reader   will also clearly see  that the discrepancy $E_u$  is identically zero in the Euclidean case, i.e. when we  view  $\mathbb{G}= \R^{n}$  as a Carnot group of step 1.      On the other hand, it turns out that  so far,  only with this  growth  assumption on $E_{u}$ that we have in \eqref{dis1} ,  strong unique continuation property (sucp)  for \eqref{e0} is known.  This  follows from the interesting  work of Garofalo and Lanconelli  ( see \cite{GLa} ) in  the case when $\mathbb{G}=\mathbb{H}^{n}$(Heisenberg group which is a Carnot group of step 2). Such a  result has been  recently generalized to Carnot groups of arbitrary step  by Garofalo and Rotz in \cite{GR}. It is to be noted that the results in \cite{GLa} and \cite{GR} follow the circle of ideas in the fundamental works \cite{GL1} and \cite{GL2}. 

\medskip

 The  purpose of our work  is to  therefore  derive sharp quantitative estimates for equations \eqref{e0} in the setup of \cite{GR} where sucp is known  so far, i.e. with the  growth assumption on the discrepancy term $E_{u}$ as in \eqref{dis1}.
Our main result Theorem \ref{main} should be seen as a subelliptic generalization  of the above mentioned Euclidean results in \cite{Bk} and \cite{Zhu1}. As the reader will realize,  such  a generalization relies on the deep link existing between the growth properties of a certain generalized Almgren frequency and  the sub-elliptic structure of $\mathbb{G}$. It turns out that in the end, they  beautifully combine. 

\medskip

In this paper , similar to \cite{Zhu1}, and \cite{BG1}, we work with an appropriate weighted version of the Almgren's Frequency  which is somewhat different from the one introduced in \cite{GR}. Having said that, we do follow  \cite{GR}  closely in parts.  Since we are interested in the question of sharp  vanishing order estimates,  it is worth emphasizing that as opposed to Theorem 7.3  in  \cite{GR},   we require some kind of  monotonicity  of the generalized  frequency that is introduced in Section 3 and not just  the  boundedness of the frequency (see Theorem \ref{mon}).  Moreover, in order to recover the  sharp vanishing order in our subelliptic situation,  we also  need to keep track of  how the  several constants that appear in  our computations   depend on the subelliptic $C^{1}$ norm of $V$ as in \eqref{vasump}  and this entails  some novel work.  As the reader will notice in Section 3, it turns out that   we have to substantially modify an argument  used in the proof of Theorem 7.3 in  \cite{GR}. This   constitutes  one of the delicate aspects of our  work and makes our   proof quite   different from that of the Laplacian as in \cite{Zhu1}   and  also from that  of the  Baouendi Grushin operators as in \cite{BG1}. 

\medskip

The  paper is organized as follows. In Section 2, we introduce some basic  notations, gather the relevant preliminary results from \cite{DG},   \cite{GR}, \cite{GLa} and  \cite{GV1} and state our main result. In Section 3, we establish  a monotonicity theorem for the generalized  weighted  Almgren type frequency that we introduce  and we then  subsequently prove our main result. 

\medskip

\textbf{Acknowledgment:} The author would like to thank  his former PhD. advisor Prof. Nicola Garofalo for introducing him to the very interesting subject of unique continuation and whose fundamental  work on this subject has been his constant inspiration. The author would also like to thank him for clarifying several results obtained in \cite{GR}.

\section{Preliminaries and Statement of main result}
\subsection{Preliminaries}
In this section, we state some preliminary results  that is relevant to our work and is similar to the one as in Section 2 in \cite{GR}. Henceforth in this paper we follow the notations adopted in \cite{GR} with a few exceptions. For most of the discussion in this section, one can find a detailed account in the book  \cite{BLU}. We recall that a Carnot group of step $h$ is a simply connected Lie group $\mathbb{G}$ whose lie algebra $\mathfrak{g}$ admits a stratification $\mathfrak{g}= V_1 \oplus ... \oplus V_h$ which is $h$ nilpotent., i.e., $[V_1, V_j]=V_{j+1}$ for  $j =1, ...h-1$ and  $[V_j, V_h]=0$ for $j =1, ...h$. A trivial example is when  $\mathbb{G}=\Rn$ and in which case $\mathfrak{g}=V_1=\Rn$. The simplest non-Abelian example of a Carnot group of step $2$ is the Heisenberg group $\mathbb{H}^{n}$, i.e. in $\R^{2n+1}$,  we let $(x, y, t)= (x_1, ...x_n, y_1, ...y_n, t)$ and  the group operation is as follows
\begin{equation}\notag
(x, y, t) \circ (x', y', t')= (x+x', y+y', t+ t'  + 2(x'.y- x.y') )
\end{equation}
In such a case, we have that  $V_1$ is spanned by 
\begin{align}\label{d1}
& X_i=\partial_{x_i} +2y_i \partial_t,\ i=1, .., n
\\
& Y_j= \partial_{y_j} - 2x_j \partial_t,\  j=1, .., n
\notag
\end{align}
and $V_2$ is spanned by $\partial_t$. 
We note  that the following holds,    
\begin{equation}\notag
[X_i, Y_j]= -4 \delta_{ij} \partial_t
\end{equation}
 and therefore $V_1$ generates the whole lie algebra.  We would like to mention that over here, we identify the lie algebra $\mathfrak{g}$ with the left invariant vector fields. 

\medskip

Now in  a Carnot group $\mathbb{G}$, by the above assumptions  on the Lie algebra, we see that any basis of the horizontal layer $V_1$ generates the whole $\mathfrak{g}$. We will respectively denote by
\begin{equation}
L_{g}(g')= gg', \\\ R_{g}(g')= g'g
\end{equation}
the left and right translation by an element $g \in \mathbb{G}$.

The exponential mapping $exp: \mathfrak{g} \to \mathbb{G}$ defines an analytic diffeomorphism onto $\mathbb{G}$. We recall the Baker-Campbell-Hausdorff formula, see for instance section 2.15 in \cite{V},
\begin{equation}
exp(c_1)exp(c_2)= exp(c_1 + c_2 + \frac{1}{2} [c_1, c_2] + \frac{1}{12} \{[c_1, [c_1, c_2]] - [c_2, [c_1, c_2]] \}+ .. )
	\end{equation}
	where the dots indicate commutators of order four and higher. Each element of the layer $V_j$ is assigned a  formal degree $j$. Accordingly, one defines dilations on $\mathfrak{g}$ by the rule 
	\begin{equation}
	\Delta_{\lambda} c= \lambda c_1 +..... \lambda^{h} c_h
	\end{equation}
	The anisotropic dilations $\delta_{\lambda}$ on $\mathbb{G}$ are then defined as 
	\begin{equation}\label{di}
	\delta_{\lambda}(g) = exp \circ \Delta_{\lambda} \circ exp^{-1} g
	\end{equation}
	Throughout the paper, we will indicate by $dg$ the bi-invariant Haar measure on $\mathbb{G}$ obtained by lifting via the exponential map exp the Lebesgue measure on $\mathfrak{g}$. Let $m_j= dim V_j$. One can check that
	\begin{equation}
	(d \circ \delta_{\lambda})(g) = \lambda^{Q}dg
	\end{equation}
	where
	$Q= \sum_{j=1}^{h} jm_j$. $Q$ is referred to as the homogeneous dimension of $\mathbb{G}$ and is in general different from the topological dimension of $\mathbb{G}$ which is $\sum_{j=1}^{h} m_j$. 
	
	\medskip
	
	We let $Z$ to be the smooth vector field which corresponds to the infinitesimal generator of the non-isotrophic dilations \eqref{di}. Note that $Z$ is characterized by the following property
	\begin{equation}
	\frac{d}{dr} u(\delta_{r} g) = \frac{1}{r} Zu(\delta_{r} g)
	\end{equation}
	Therefore if $u$ is homogeneous of degree $k$ with respect to \eqref{di}, i.e., $u(\delta_{r} g)= r^{k} u(g)$, then we have that $Zu= ku$.	
	\medskip
	
	Let $ \{ e_1, ...,e_m \}$ be an orthonormal basis of the first layer $V_1$ of the Lie algebra. We define the  corresponding left invariant smooth vector fields by the formula
	\begin{equation}
	X_{i}(g)=dL_{g}(e_i), \\\  i=1, ..., m
	\end{equation}	where $dL_{g} $ denote  the differential of $L_{g}$. We assume that $\mathbb{G}$ is endowed with a left-invariant Riemannian metric such that $\{X_1, ...., X_m \}$ are orthonormal.  We note that in this case, the bi-invariant Haar measure $dg$ agrees with the Riemannian volume element ( see for instance \cite{BLU}). The corresponding subLaplacian is defined by the formula
	
	\begin{equation}\label{Def}
	\Delta_{H} u = \sum_{i=1}^m X_{i}^2 u
	\end{equation}
	We note that by Hormander's theorem, $\Delta_{H}$ is hypoelliptic. We with  indicate with $e$ the identity element of $\mathbb{G}$. 
	
	\medskip
	
	Let $\Gamma (g, g')= \Gamma (g', g)$ be the positive fundamental solution of $-\Delta_{H}$. It turns out that $\Gamma$ is left invariant, i.e., 
	\begin{equation}
	\Gamma (g, g')= \tilde{\Gamma}(g^{-1} \circ g')
	\end{equation}
	
	For every $r>0$, let 
	\begin{equation}
	B_{r}= \{ g \in \mathbb{G} | \Gamma (g, e) > \frac{1}{r^{Q-2}} \}
	\end{equation}
	It was proved by Folland  \cite{F1} that $\tilde{\Gamma}(g)$ is homogeneous of order $2-Q$ with respect to the non-isotrophic dilations \eqref{di}. Therefore, if we define
	\begin{equation}
	\rho(g)= \Gamma(g)^{\frac{-1}{Q-2}}
	\end{equation}
	then $\rho$ is homogenous of degree 1.  One can immediately see that $B_r$ can be equivalently characterized as 
	\begin{equation}
	B_{r}= \{ g:\rho(g) < r \}
	\end{equation}
	
	We let $S_{r}= \partial B_r$. We note that since $\Gamma$ is homogeneous of degree $2-Q$, therefore
	\begin{equation}\label{gm}
	Z\Gamma= (2-Q) \Gamma
	\end{equation}
	Now by the  strong maximum principle (Since $\Gamma(g, e)$ is harmonic for $g \neq e$), we have that $\Gamma(g, e) > 0$ for all $ g \neq e$. Now since $Z \Gamma= <D\Gamma, Z>$, where $D\Gamma$ is the Riemannian gradient with respect to the left invariant metric, we conclude from \eqref{gm} that $D\Gamma$ never vanishes. Therefore, by implicit function theorem, we conclude that the level sets $S_{r}$ are smooth hypersurfaces in $\mathbb{G}$.
	
	 The position $d(g, g')$ defined by
	\begin{equation}
	d(g, g')= \rho(g^{-1} \circ g')
	\end{equation}
	defines a pseudo-distance on $\mathbb{G}$.    In what follows, we denote by 
	\begin{equation}
	\nabla_{H}u= \sum_{i=1}^m X_i u X_i
	\end{equation}	
	the horizontal gradient of $u$. We also let
	\begin{equation}
	|\nabla_{H}u|^2= \sum_{i=1}^m (X_i u)^2
	\end{equation}	
	
	\subsection{Statement of the main result}
	In order to state our main result, we first describe our  framework.  We will assume that $u$ is   a solution to 
	\begin{equation}\label{E0}
	\Delta_{H} u  = Vu
	\end{equation}
	in $B_1$. Since the regularity issues are not our main concern, we will assume  apriori that  $u, X_{i}u, X_{i}X_{j} u, Zu$   are in $L^{2}(B_1)$ with respect to the Haar measure $dg$.    Concerning the potential $V$, we assume that  it  satisfies 
	\begin{equation}\label{vasump}
|V| \leq K  |\nabla_{H} \rho|^2 \ \ \ \ \ |ZV| \leq K |\nabla_{H} \rho|^2
\end{equation}	
for some $K>1$.  An example of  a smooth $V$ which satisfies \eqref{vasump} is given by $ V(g) = \tilde{V}(g) f(\rho|\nabla_{H} \rho|^2)$ where $\tilde{V}$ is a smooth function defined on $\mathbb{G}$ and $f:\R \to \R$ is a smooth compactly  supported function which vanishes in a neighborhood of $0$. The first condition in \eqref{vasump} is easy to see and for the second condition we note that
\begin{equation}\label{above}
|ZV| \leq |(Z\tilde{V}) f(\rho |\nabla_{H} \rho|^2)| + |\tilde{V} f'( \rho|\nabla_{H} \rho|^2) Z\rho |\nabla_{H}\rho|^2|  + |\tilde{V} f'(\rho |\nabla_{H} \rho|^2) \rho Z(|\nabla_{H} \rho|^2)|
\end{equation}
Now since $\rho$ has homogeneity $1$, $Z\rho=\rho$ and since $|\nabla_{H} \rho|^2$ has homogeneity 0 being the derivative of 1 homogenous function, we have that $Z(|\nabla_{H}\rho|^2)=0$ for $g \neq e$. We note that the derivative in \eqref{above} is only computed for $g \neq e$ since $f$ vanishes in a neighborhood of $0$.  Therefore we obtain
\begin{equation}\label{obt}
|ZV| \leq |(Z\tilde{V}) f(\rho |\nabla_{H} \rho|^2)| + |\tilde{V} f'( \rho|\nabla_{H} \rho|^2) \rho |\nabla_{H}\rho|^2|
\end{equation}

 From \eqref{obt}, it is easy to see  that the second condition in \eqref{vasump} is satisfied for this choice of $V$. 

As in \cite{GR}, we define the discrepancy $E_{u}$  at $e$ by
\begin{equation}\label{dis}
E_{u}= < \nabla_{H} u, \nabla_{H} \rho> - \frac{Zu}{\rho} |\nabla_{H}  \rho|^2
\end{equation}
Like in \cite{GR}, we will also assume that 
\begin{equation}\label{dis1}
|E_u| \leq \frac{f(\rho)}{\rho} |\nabla_{H} \rho|^2 |u|
\end{equation}
where $f: (0, 1) \to (0, \infty)$ is a continuous  increasing function which satisfies the Dini integrability condition
\begin{equation}\label{Dini}
\int_{0}^1 \frac{f(t)}{t} dt <  K_0, \\\  |f| \leq K_1
\end{equation}
We now    list a few classes of examples from \cite{GR} in which the assumption \eqref{dis1} holds.

\medskip

In the case when $\mathbb{G}= \R^n$, we have that $Z= \Sigma_{i=1}^n x_i \partial_i$ and $\nabla_{H} \rho= \frac{x}{|x|}$. Therefore,  we clearly  see that $E_u \equiv 0$ in this case. 

\medskip

 For a general  Carnot group $\mathbb{G}$, when  $u$ is radial, i.e., if    $u(g)= f( \rho(g))$, then  it follows from a  straightforward calculation (See Proposition 9.6 in \cite{GR}) that $E_u \equiv 0$.   

\medskip

Also, if we specialize to the case when  $\mathbb{G} = \mathbb{H}^n$ and  assume  $u$ to be  polyradial, i.e. with $g=( w_1,...., w_n, t)$ where $ w_{i}= (x_i, y_i)$, we have that $u(g)= \phi (|w_1|, ..., |w_n|, t)$, then $E_u \equiv 0$. ( see Proposition 9.11 in \cite{GR}).  It is however not true that for general groups of Heisenberg type ( see Section 9.1 in \cite{GR} for the precise definition of groups of Heisenberg type), polyradial functions have zero discrepancy. ( see Section 9 in \cite{GR} for a counterexample) Nevertheless,  given these   examples, we would like to think of the growth condition \eqref{dis1} on $E_u$  as the measure of a certain symmetry type property of $u$. 

\medskip

We now  state our main result.

\begin{thrm}\label{main}
Let $u$  be a solution to \eqref{E0} in $B_1$  such that $|u| \leq C_0$ and the discrepancy $E_u$ of $u$ at $e$  satisfies \eqref{dis1}. Let $V$ satisfy \eqref{vasump}. Then there exists a universal $a \in (0, 1/3)$, and constants $C_1, C_2$ depending on $Q, C_0$  and  $K_0, K_1$ in \eqref{Dini}  and  also $\int_{B_{1/3} }u^2 |\nabla_{H} \rho|^2  dg$ such that for all $0 <r  < a$, one has
\begin{equation}\label{est}
||u||_{L^{\infty}(B_r)} \geq C_1 r^{C_2 \sqrt{K}}
\end{equation}
\end{thrm}
\begin{rmrk}
We note that in the elliptic case as in \cite{BG2}, \cite{Zhu1}, we have that $|\nabla_{H} \rho|^2 \equiv 1$ and in such a case, the constant $K$ in \eqref{vasump} can be taken to be $C(||V||_{W^{1, \infty} } + 1)$ for some universal $C$. We  thus see that   in the elliptic case, 
\eqref{est} reduces to the Euclidean result as in \cite{Bk} and \cite{Zhu1} since $E_u \equiv 0$ in the Euclidean case. Therefore   our estimate \eqref{est} gives  sharp  bounds on  the vanishing order  of $u$ at the identity $e$ in terms of  a  certain "subelliptic"  $C^{1}$ norm of $V$  with bounds as  in \eqref{vasump}.  Hence, our result can be  thought of as a subelliptic analogue of the sharp quantitative uniqueness result  in the Euclidean case. 
\end{rmrk}
\begin{rmrk}
It is worth mentioning the case when  $\mathbb{G} = \mathbb{H}^n$ and $E_u \equiv 0$.  Now  from the definition of the sublaplacian and the explicit  representation  of the horizontal  vector fields  for $\mathbb{H}^{n}$  as in \eqref{d1},  we have that $u$ solves the following equation 
\begin{equation}\label{Car}
\Delta_{z} u + \frac{|z|^2}{4} \partial_{tt} u  + \partial_t  \theta_{u} = Vu
\end{equation}
where $\theta_u = \Sigma (x_j \partial_{y_j}u - y_j \partial_{x_j}u)$ and $E_u= \frac{4}{\rho^3} t\theta_u$ ( see for instance Lemma 9.8 in \cite{GR}). Now when  $E_u \equiv 0$, we get that $\theta_u \equiv 0$ and hence $u$ solves the stationary Schr\"odinger equation  corresponding to the  Baouendi-Grushin operator 
\begin{equation}
\Delta_{z} u + \frac{|z|^2}{4} \partial_{tt} u = Vu
\end{equation}
for which the  sharp quantitative estimate \eqref{est} follows from Theorem 1.1 in \cite{BG1}.  

\medskip

   However  if we only assume that $E_u$  satisfies \eqref{dis1}, then   from  \eqref{dis1} and the fact that $E_u= \frac{4}{\rho^3} t \theta_u$ we  can only assert that $u$ solves  \eqref{Car} where $\theta_u$ satisfies  the following growth condition, 
\begin{equation}
|\theta_u| \leq \frac{ f(\rho) \rho^2 |\nabla_{H} \rho|^2 |u|}{4 t}
\end{equation}
and  in this case,  our result  is not implied by \cite{BG1}. Therefore  our result  is new even for $\mathbb{G}=\mathbb{H}^{n}$.  

\end{rmrk}

\begin{rmrk} 
It remains to be  seen  when the potential $V$ only  satisfies
\begin{equation}
|V| \leq K |\nabla_{H} \rho|^2
\end{equation}
instead of \eqref{vasump}, then if  it can be shown that   the vanishing order of the solution $u$  is bounded from above  by $C K^{2/3}$. This  would constitute  the subelliptic  analogue of the result in \cite{K} to which we would  like to  come back in a future  study.
\end{rmrk}

\section{Proof of Theorem \ref{main}}
\subsection{Monotonicity of a generalized frequency}
Following \cite{Zhu1} and \cite{BG1}, for $\alpha > 0$ to be decided later,  we let
\begin{equation}\label{h}
H(r)= \int_{B_r} u^2 |\nabla_{H} \rho|^2 (r^2 - \rho^2)^{\alpha} dg
\end{equation}
For notational convenience, we will let $|\nabla_{H} \rho|^2= \psi$.
Therefore with this new notation, we have that 
\begin{equation}
H(r)= \int_{B_r} u^2 (r^2 - \rho^2)^{\alpha} \psi
\end{equation}

By differentiating with respect to $r$, we get that
\begin{equation}
H'(r) = 2 \alpha r \int u^2 (r^2 - \rho^2)^{\alpha -1} \psi
\end{equation}
Using the identity
\begin{equation}
(r^2 - \rho^2)^{\alpha -1} = \frac{1}{r^2} (r^2 - \rho^2)^{\alpha} + \frac{\rho^2}{r^2} (r^2 - \rho^2)^{\alpha -1}
\end{equation}
the latter equation can be rewritten as
\begin{equation}
H'(r) = \frac{2 \alpha}{r} H(r) + \frac{2 \alpha} {r} \int u^2 (r^2 - \rho^2)^{\alpha-1} \rho^2 \psi
\end{equation}
Now by using the fact that  $Z \rho = \rho$, we see that  $(r^2 - \rho^2)^{\alpha-1} \rho^2$ can be rewritten as 
\begin{equation}
(r^2 - \rho^2)^{\alpha-1} \rho^2= -\frac{1}{2 \alpha} Z(r^2- \rho^2)^{\alpha}
\end{equation}
 Therefore we get that
\begin{equation}
H'(r)= \frac{2 \alpha}{r} - \frac{1}{r} \int u^2 Z(r^2-\rho^2)^{\alpha} \psi
\end{equation}
Now we note that the following  two identity holds
\begin{equation}\label{l1}
 Z( |\nabla_{H} \rho|^2)=0,\ g\neq e
 \end{equation}
  and
\begin{equation}\label{l2}
div_{\mathbb{G}} Z = Q
\end{equation}

For \eqref{l2},  for instance the reader can refer to \cite{DG}. Note that over here, $div_{\mathbb{G}}$ denotes the Riemmanian divergence on $\mathbb{G}$.  Now by using the Divergence theorem on $\mathbb{G}$ with respect to its Riemmanian structure and  also by using \eqref{l1}, \eqref{l2}, we get that
\begin{equation}\label{v1}
H'(r)= \frac{2 \alpha + Q}{r} H(r) + \frac{2}{r} \int u Zu (r^2 - \rho^2)^{\alpha} \psi 
\end{equation}
Over here, we  crucially use the fact that since $|\nabla_{H} \rho|^2$ has homogeneity $0$, therefore it is bounded and hence the integration by parts can be justified by an approximation type argument. Now by using \eqref{dis1}, we get that
\begin{equation}\label{h3}
H'(r) = \frac{2 \alpha + Q}{r} H(r)  + \frac{2} {r} \int u \rho < \nabla_{H} u, \nabla_{H} \rho> (r^2 - \rho^2)^{\alpha}  +K(r)
\end{equation}
where
\begin{equation}
|K(r)| \leq \frac{f(r)}{r} H(r)
\end{equation}
\eqref{h3} can hence  be rewritten as 
\begin{equation}\label{v}
H'(r)= \frac{2 \alpha+ Q}{r} H(r) + \frac{1}{(\alpha+1)r} I(r) + K(r)
\end{equation}
where 
\begin{equation}\label{i}
I(r)= 2 (\alpha +1 ) \int u < \nabla_{H} u, \nabla_{H} \rho> (r^2- \rho^2)^{\alpha} \rho= -\int u <\nabla_{H} u, \nabla_{H} (r^2 - \rho^2)^{\alpha+1}>
\end{equation}
Now we note that the following identity holds( see for instance \cite{GV1})
\begin{equation}\label{df}
div_{\mathbb{G}}X_i=0
\end{equation}

Therefore, by  applying  integrating by parts to \eqref{i} and by using the equation \eqref{E0} and the identity \eqref{df} we get that,
\begin{equation}\label{i1}
I(r) = \int |\nabla_{H} u|^2 (r^2 - \rho^2)^{\alpha+1} + Vu^2 (r^2-\rho^2)^{\alpha+1}
\end{equation}
 
We now  define the generalized  frequency of $u$ as 
\begin{equation}\label{freq}
N(r)= \frac{I(r)}{H(r)}
\end{equation}
The central result of this section which implies our main estimate \eqref{est} in Theorem \ref{main} is the following monotonicity result of $N(r)$.
\begin{thrm}\label{mon}
For $\alpha = \sqrt{K}$, we have that there  exists universal $\overline{C}$ depending on  $Q, K_0, K_1$ such that
\begin{equation}\label{m1}
r \to e^{\overline{C} \int_{0}^r \frac{f(t)}{t} }(N(r) + \overline{C} K ( r^2 +  \int_{0}^r \frac{f(t)}{t} dt))
\end{equation}
is monotone increasing on $(0, 1)$.
\end{thrm}
\begin{proof}
The proof will be divided into several steps. We first calculate $I'(r)$.
By differentiating the expression in \eqref{i1} with respect to $r$, we get that
\begin{equation}\label{m}
I'(r)= 2(\alpha+1)r \int |\nabla_{H} u|^2 (r^2 - \rho^2)^{\alpha} + 2(\alpha+1)r \int Vu^2 (r^2-\rho^2)^{\alpha}
\end{equation}
This can be rewritten as
\begin{equation}\label{i'}
I'(r) = \frac{2(\alpha+1)}{r} \int |\nabla_{H} u|^2 (r^2 - \rho^2)^{\alpha+1}  + \frac{2(\alpha+1)}{r} \int |\nabla_{H} u|^2 (r^2 - \rho^2)^{\alpha} \rho^2 + 2(\alpha+1)r \int Vu^2 (r^2-\rho^2)^{\alpha}
\end{equation}
Using the fact that $Z\rho=\rho$,  the second term  on the right hand side  of above expression can be rewritten as
\begin{equation}
\frac{2(\alpha+1)}{r} \int |\nabla_{H} u|^2 (r^2 - \rho^2)^{\alpha} \rho^2 = -\frac{1}{r} \int |\nabla_{H} u|^2 Z(r^2 - \rho^2)^{\alpha+1}
\end{equation}
At this point, we need the following Rellich type identity  which corresponds to   Corollary 3.3 in \cite{GV1}. This can can be thought of as the sub-elliptic analogue of Rellich type identity established in \cite{PW}. For a $C^{1}$ vector  field $F$,  we have that 
 \begin{align}\label{re}
&  2 \int_{\partial B_r}Fu < \nabla_{H}u, N_{H}>  dH^{n-1} + \int_{B_r} div_{\mathbb{G}} F |\nabla_{H} u|^2 dg
\\
& -  2\int_{B_r} X_i u [X_i, F]u dg   - 2 \int_{B_r} Fu \Delta_{\mathbb{H}} u dg
\notag\\
& = \int_{\partial B_r} |\nabla_{H} u|^2 < F, \nu> dH^{n-1}
\notag
\end{align}
We now apply the identity \eqref{re} to the vector field $ F= (r^2- \rho^2)^{\alpha+1} Z$. We note that the boundary terms don't appear due to the presence of the weight $(r^2- \rho^2)^{\alpha+1}$. Therefore we get,

\begin{equation}
-\frac{1}{r} \int |\nabla_{H} u|^2 Z(r^2 - \rho^2)^{\alpha+1}
= -\frac{1}{r} \int |\nabla_{H} u|^2 div_{\mathbb{G}} (F) + \frac{Q}{r} \int |\nabla_{H} u|^2 (r^2 - \rho^2)^{\alpha+1}
\end{equation}
where we used the fact that $div_{\mathbb{G}} Z=Q$. Now by applying \eqref{re}, we get that
\begin{equation}\label{i10}
-\frac{1}{r} \int |\nabla_{H} u|^2 Z(r^2 - \rho^2)^{\alpha+1}
= -\frac{2}{r} \int X_i u [X_i, F] u dg - \frac{2}{r} \int Fu \Delta_{H} u dg + \frac{Q}{r} \int |\nabla_{H} u|^2 (r^2 - \rho^2)^{\alpha+1}
\end{equation}
At this point, we  note that the following  identity holds ( See for instance \cite{DG})
\begin{equation}\label{li}
[X_i, Z] = X_i
\end{equation}
 Therefore by using \eqref{li},  we have 
\begin{equation}\label{li1}
[X_i, F]u = X_{i}(r^2- \rho^2)^{\alpha +1 } Z + (r^2 - \rho^2)^{\alpha+1} X_i= -2(\alpha+1) \rho (r^2 - \rho^2)^{\alpha} X_i \rho Z + (r^2 - \rho^2)^{\alpha+1} X_i
\end{equation}
By using \eqref{li1} in \eqref{i10} we get that,
\begin{align}
&- \frac{1}{r} \int |\nabla_{H} u|^2 Z(r^2 - \rho^2)^{\alpha+1} = \frac{4(\alpha+1)}{r} \int < \nabla_{H} u, \nabla_{H} \rho> \rho Zu (r^2 - \rho^2)^{\alpha} 
\\
& - \frac{2}{r}  \int Vu Zu  (r^2 - \rho^2)^{\alpha+1}  +  \frac{Q-2}{r} \int |\nabla_{H} u|^2 (r^2 - \rho^2)^{\alpha+1}
\notag
\end{align}
Now by using  the  growth assumption \eqref{dis1}  on the discrepancy $E_u$ we get that
\begin{align}
&- \frac{1}{r} \int |\nabla_{H} u|^2 Z(r^2 - \rho^2)^{\alpha+1} = \frac{4(\alpha+1)}{r} \int  (Zu)^2 (r^2 - \rho^2)^{\alpha} \psi + \frac{Q-2}{r} \int |\nabla_{H} u|^2 (r^2 - \rho^2)^{\alpha+1}
\\
& - \frac{2}{r}  \int Vu Zu  (r^2 - \rho^2)^{\alpha+1}  + K_1(r)
\notag
\end{align}
where 
\begin{equation}
|K_1 (r)| \leq \frac{4(\alpha+1) f(r)}{r}  \int (r^2 - \rho^2)^{\alpha} |u| |Zu| \psi  
\end{equation} 
Therefore by substituting the above expression in  \eqref{i'} we get that,
\begin{align}
& I'(r)= \frac{2\alpha +Q }{r} \int |\nabla_{H} u|^2 (r^2 - \rho^2)^{\alpha+1}  + \frac{4(\alpha+1)}{r} \int  (Zu)^2 (r^2 - \rho^2)^{\alpha} \psi 
\\
&  + 2(\alpha+1)r \int Vu^2 (r^2-\rho^2)^{\alpha} - \frac{2}{r}  \int Vu Zu  (r^2 - \rho^2)^{\alpha+1}  + K_1(r)
\notag
\end{align}
Recalling the definition of $I(r)$, we  can   rewrite $I'$ as 
\begin{align}\label{i11}
& I'(r) = \frac{2 \alpha + Q}{r} I(r) - \frac{2 \alpha + Q}{r} \int Vu^2 (r^2-\rho^2)^{\alpha+1} +  \frac{4(\alpha+1)}{r} \int  (Zu)^2 (r^2 - \rho^2)^{\alpha} \psi 
\\
&  + 2(\alpha+1)r \int Vu^2 (r^2-\rho^2)^{\alpha} - \frac{2}{r}  \int Vu Zu  (r^2 - \rho^2)^{\alpha+1}  + K_1(r)
\notag
\end{align}
Now by integrating by parts and again by using the fact that $div_{\mathbb{G}} Z=Q$, we get that
\begin{align}\label{i12}
&  - \frac{2}{r}  \int Vu Zu  (r^2 - \rho^2)^{\alpha+1} = -\frac{1}{r} \int Z(u^2) V (r^2 - \rho^2)^{\alpha+1}
\\
& = \frac{1}{r} \int  u^2 div_{\mathbb{G}}( (r^2-\rho^2)^{\alpha+1} VZ)= \frac{Q}{r} \int Vu^2 (r^2-\rho^2)^{\alpha+1}
\notag
\\
& + \frac{1}{r} \int u^2 ZV (r^2 - \rho^2)^{\alpha+1}  -\frac{2(\alpha+1)}{r} \int Vu^2 \rho^2 (r^2-\rho^2)^{\alpha}
\notag
\end{align}
At this point, we note from \eqref{vasump}  that the following estimate holds
\begin{equation}\label{i13}
|\frac{Q}{r} \int Vu^2 (r^2-\rho^2)^{\alpha+1}| \leq CKr H(r)
\end{equation}
and also
\begin{equation}\label{j13}
|\frac{1}{r} \int u^2 ZV (r^2 - \rho^2)^{\alpha+1}| \leq CKr H(r)
\end{equation} 
for some universal $C$. 
We now write  the expression $\frac{2 \alpha + Q}{r} \int Vu^2 (r^2-\rho^2)^{\alpha+1}$ as
\begin{equation}\label{i14}
\frac{2 \alpha + Q}{r} \int Vu^2 (r^2-\rho^2)^{\alpha+1}= (2 \alpha+ Q)r \int Vu^2 (r^2 -\rho^2)^{\alpha}  - \frac{2\alpha+Q}{r} \int Vu^2 (r^2-\rho^2)^{\alpha} \rho^2
\end{equation}
Therefore, by using \eqref{i12}, \eqref{i13}, \eqref{j13} and \eqref{i14} in \eqref{i11} and also  by using \eqref{vasump}  we get that
\begin{equation}\label{i20}
 I'(r)= \frac{2 \alpha +Q}{r}I(r) + \frac{4(\alpha+1)}{r} \int (Zu)^2 (r^2 -\rho^2)^{\alpha} \psi + O(1) Kr H(r) + K_1(r)
\end{equation}

Finally from the definition of $N(r)$ as in \eqref{freq} and   from \eqref{v1} and \eqref{i20}  we get  that  the following inequality holds
\begin{align}\label{calc1}
& N'(r) = \frac{I'(r)}{H(r)} - \frac{H'(r)}{H(r)} N(r)
\\
& \geq    -C_1K r 
\notag
\\
& + ( \frac{4(\alpha+1) \int (Zu)^2 (r^2-\rho^2)^{\alpha} \psi}{ r H(r)}  - \frac{4(\alpha+1)(\int (r^2-\rho^2)^{\alpha} u Zu \psi )(\int (r^2- \rho^2)^{\alpha}u <\nabla_{H}\rho, \nabla_{H}u>\rho)}{rH^2(r)}
\notag
\\
& - \frac{4(\alpha+1)f(r) (\int |u| |Zu| (r^2- \rho^2)^{\alpha} \psi)}{r H(r)}
\notag
\end{align}
where $C_1$ is universal. Now by using \eqref{dis1}, we get that
\begin{equation}\label{calc2}
4(\alpha+1) \int (r^2- \rho^2)^{\alpha}u <\nabla_{H}\rho, \nabla_{H}u>\rho= 4(\alpha+1) \int u Zu (r^2- \rho^2)^{\alpha} \psi + K_2 (r)
\end{equation}
where
\begin{equation}\label{calc3}
|K_2(r)| \leq 4(\alpha + 1) f(r) H(r)
\end{equation}
Therefore by using  \eqref{calc2} and \eqref{calc3} in \eqref{calc1} we get that
\begin{align}\label{calc}
& N'(r) = \frac{I'(r)}{H(r)} - \frac{H'(r)}{H(r)} N(r)
\\
& \geq    -C_1K r 
\notag
\\
& + ( \frac{4(\alpha+1) \int (Zu)^2 (r^2-\rho^2)^{\alpha} \psi}{ r H(r)}  - \frac{4(\alpha+1)(\int (r^2-\rho^2)^{\alpha} u Zu \psi )^2}{rH^2(r)}
\notag
\\
& - \frac{8(\alpha+1)f(r) (\int |u| |Zu| (r^2- \rho^2)^{\alpha} \psi)}{r H(r)}
\notag
\end{align}

At this point, we need a  modified form of an  argument  used  in the proof of Theorem 7.3  in \cite{GR}. Before  proceeding further, we make the following remark.
\begin{rmrk}
In the subsequent expressions, all the constants $C_i$'s, $\tilde{C_i}$'s that will  appear are all  universal and only depends on $C_0, Q $ and $K_0, K_1$ as in \eqref{Dini}.
\end{rmrk}
 Note that from the definition of $E_u$ and the growth condition \eqref{dis} that the following holds
\begin{equation}\label{in2}
\int u Zu (r^2 - \rho^2)^{\alpha} \psi = \frac{I(r)}{ 2(\alpha+1)} + H_{1} (r)
\end{equation}
where 
\begin{equation}\label{bds}
|H_{1} (r)| \leq f(r)  H(r)
\end{equation}

Now from the expression  of $I(r)$ as in \eqref{i1}  and  also from the assumption on $V$ as in \eqref{vasump}, we have that 
\begin{equation}\label{i100}
I(r) + Kr^2 H(r) \geq 0
\end{equation}
Since $\alpha = \sqrt{K}$,  by taking into account \eqref{i100} we get that
\begin{equation}\label{in1}
\frac{I(r)}{ 2 (\alpha +1)}  + \sqrt{K} r^2 H(r) \geq 0
\end{equation}
By substituting \eqref{in1} in \eqref{in2}, we obtain
\begin{equation}
\int u Zu (r^2 - \rho^2 )^{\alpha} \psi - H_1(r) + \sqrt{K} r^2 H(r) \geq   0
\end{equation}
 Now  because of the bound in \eqref{bds} we get from the above inequality that the following holds
\begin{equation}\label{in3}
\int u Zu (r^2 - \rho^2)^{\alpha}\psi  +  \sqrt{K} r^2 H(r)  + f(r) H(r) \geq 0
\end{equation}
We now distinguish 2 cases. Either we have that

\textbf{Case 1}:
\begin{equation}\label{c1}
(\int u^2 (r^2 - \rho^2)^{\alpha} \psi)^{1/2} ( \int (Zu)^2 (r^2 - \rho^2)^{\alpha} \psi )^{1/2} \leq \sqrt{2} ( \int u Zu (r^2 - \rho^2 )^{\alpha} \psi + 8  \sqrt{K} r^2 H(r) + f(r) H(r))
\end{equation}

or

\medskip

\textbf{Case 2}:
\begin{equation}\label{c2}
(\int u^2 (r^2 - \rho^2)^{\alpha} \psi)^{1/2} ( \int (Zu)^2 (r^2 - \rho^2)^{\alpha} \psi )^{1/2} > \sqrt{2} ( \int u Zu (r^2 - \rho^2 )^{\alpha} \psi + 8 \sqrt{K} r^2 H(r)  + f(r) H(r))
\end{equation}
If  \textbf{Case} 1  ( i.e. \eqref{c1} ) occurs, then  by applying  Cauchy-Schwartz inequality  to  the expression
\begin{align}
& ( \frac{4(\alpha+1) \int (Zu)^2 (r^2-\rho^2)^{\alpha} \psi}{ r H(r)}  - \frac{4(\alpha+1)(\int (r^2-\rho^2)^{\alpha} u Zu \psi )^2}{ r H(r)^2} )
\notag
\end{align}
 in \eqref{calc},  we see that the above expression is non-negative.

Now we estimate the term 
\begin{align}
&- \frac{8(\alpha+1)f(r) (\int |u| |Zu| (r^2- \rho^2)^{\alpha} \psi)}{r H(r)} 
\notag
\end{align}
  in \eqref{calc} by using  Cauchy-Schwartz and  also  by using the estimate  \eqref{c1} and consequently obtain  
\begin{align}\label{in101}
& |\frac{8(\alpha+1)f(r) (\int |u| |Zu| (r^2- \rho^2)^{\alpha} \psi)}{r H(r)}| 
\\
& \leq \frac{8(\alpha+1) f(r) (\int u Zu (r^2 - \rho^2 )^{\alpha} \psi + 8  \sqrt{K} r^2 H(r) + f(r) H(r))}{r H(r)}
\notag
\end{align}
Now from \eqref{in2} we get that
\begin{equation}\label{in100}
|\frac{8(\alpha+1) f(r)\int u Zu (r^2 - \rho^2 )^{\alpha} \psi} {rH(r)}|  \leq 4 \frac{f(r)}{r} N(r) + \tilde{C} K \frac{f^2(r)}{r} 
\end{equation}
Therefore by using \eqref{in100}  in \eqref{in101} we have that the following holds
\begin{align}\label{calc5}
& N'(r) \geq  - C_2 \frac{f(r)}{r} N(r)   -C_3 K r 
\\
& -  \tilde{C_1}  K \frac{f^{2}(r)}{r} -\tilde{C_2} K \frac{f(r)}{r}
\notag
\end{align}
In \eqref{calc5}, we crucially used the fact that $\alpha=\sqrt{K} \leq K$.  Now since  $|f|\leq K_1$ we  obtain from \eqref{calc5} that the following holds
\begin{equation}\label{ct1}
N'(r) \geq  - C_2 \frac{f(r)}{r} N(r)   -C_4 K r -  C_5 K  \frac{f(r)}{r}
\end{equation}
If instead \textbf{Case} 2 (i.e.\eqref{c2} ) occurs, then there are 2 sub-cases. Either

\textbf{subcase 1}
\begin{equation}\label{c10}
\int u Zu (r^2-\rho^2)^{\alpha} \psi \geq 0
\end{equation}
or

\medskip

\textbf{subcase 2}
\begin{equation}\label{c11}
\int u Zu (r^2-\rho^2)^{\alpha} \psi \leq 0
\end{equation}  
If \textbf{subcase} 1 (i.e.\eqref{c10})  occurs, then the argument is relatively simple.  By using the inequality $(a+b)^{2} \geq a^2 $ when $a, b \geq 0$  in \eqref{c2} we get that
\begin{equation}\label{in7}
(\int u^2 (r^2 - \rho^2)^{\alpha} \psi)( \int (Zu)^2 (r^2 - \rho^2)^{\alpha} \psi ) \geq 2 (\int u Zu (r^2 - \rho^2)^{\alpha} \psi )^2  
\end{equation}
From \eqref{in7}, it follows that
\begin{align}
& \frac{4(\alpha+1) \int (Zu)^2 (r^2-\rho^2)^{\alpha} \psi}{ r H(r)}  - \frac{4(\alpha+1)(\int (r^2-\rho^2)^{\alpha} u Zu \psi )^2}{ r H(r)^2}
\\
& \geq \frac{2(\alpha+1) \int (Zu)^2 (r^2-\rho^2)^{\alpha} \psi}{ r H(r)}
\end{align}
By  using the above inequality  in \eqref{calc}, we get that
\begin{align}\label{calc6}
& N'(r) \geq    -C_1 K r 
\\
&  + \frac{2(\alpha+1) \int (Zu)^2 (r^2-\rho^2)^{\alpha} \psi}{ r H(r)} -  \frac{8(\alpha+1)f(r) (\int |u| |Zu| (r^2- \rho^2)^{\alpha} \psi)}{r H(r)}
\notag
\end{align}
Now by applying Cauchy-Schwartz inequality with $\ve$, i.e. the inequality 
\begin{equation}
2ab \leq  \ve a^2 + \frac{b^2}{\ve}
\end{equation}

  to the term  
  \[
  \frac{8(\alpha+1)f(r) (\int |u| |Zu| (r^2- \rho^2)^{\alpha} \psi)}{r H(r)}
  \]
  
    in \eqref{calc6} for small enough $\ve$, we get that
\begin{align}\label{calc7}
& N'(r) \geq  - C_1 K r  -C_6 K \frac{f(r)}{r}
\end{align}
 If instead \textbf{subcase} 2( i.e. \eqref{c11} ) occurs,  then we first note that \eqref{c2} trivially implies that

\begin{equation}\label{j100}
(\int u^2 (r^2 - \rho^2)^{\alpha} \psi)^{1/2} ( \int (Zu)^2 (r^2 - \rho^2)^{\alpha} \psi )^{1/2} > \sqrt{2} ( \int u Zu (r^2 - \rho^2 )^{\alpha} \psi +  \sqrt{K} r^2 H(r) + f(r) H(r))
\end{equation}
Now by squaring the above inequality  in \eqref{j100}( where we taking into account that the right hand side in the above inequality is non-negative due to \eqref{in3}) and then by using $(a+b)^2 \geq a^2 + 2ab$  for $b \geq 0$  with
\begin{align}
& a= \int u Zu (r^2 - \rho^2)^{\alpha}\psi
\\
& b= \sqrt{K} r^2 H(r) + f(r) H(r)
\notag
\end{align}
we get that

\begin{align}\label{pin13}
&  H(r) ( \int (Zu)^2 (r^2 - \rho^2)^{\alpha} \psi ) \geq  2( \int u Zu (r^2 - \rho^2 )^{\alpha} \psi)^2 
\\
&  + 4 \sqrt{K} r^2 (\int u Zu (r^2- \rho^2)^{\alpha} \psi) H(r) 
\notag
\\
&     + 4 f(r) (\int u Zu (r^2- \rho^2)^{\alpha} \psi) H(r)
\notag
 \end{align}
 Now we note that \eqref{in3} and \eqref{c11} together imply that
 \begin{equation}\label{jn3}
- \sqrt{K}r ^2 H(r) - f(r) H(r) \leq \int u Zu (r^2 - \rho^2)^{\alpha} \psi \leq 0
\end{equation}

  Therefore  by using   \eqref{pin13} in \eqref{calc} and then  by subsequently using the  estimate \eqref{jn3} we get that, 
 \begin{align}\label{pin4}
 & N'(r) \geq    -C_7 K r 
 \\
 & + \frac{2(\alpha+1) \int (Zu)^2 (r^2-\rho^2)^{\alpha} \psi}{ r H(r)}  - 16 K^{3/2}r^3  
 \notag
 \\
 & - C_{8} K \frac{f(r)}{r} - \frac{8(\alpha+1)f(r) (\int |u| |Zu| (r^2- \rho^2)^{\alpha} \psi)}{r H(r)}
 \notag
  \end{align} 
     where we   used the fact that $|f| \leq  K_1$ and $\alpha=\sqrt{K}$.   In order to get to \eqref{pin4}, we    also used the fact that  $(\alpha +1 ) \leq 2 \alpha$ since $\alpha > 1$.
     
     \medskip
     
     Now  if we consider   the term $16 K^{3/2} r^3$  in \eqref{pin4}, we note that it appears with a negative sign on the right hand side  and  the exponent of $K$ in that term is $\frac{3}{2}$  which is more than 1. This would not let us  conclude the desired monotonicity result in  \eqref{m1}. Therefore, we have  to get rid of this term in the final expression of $N'$.   In order  to do so,  we first note that since \eqref{in3} holds, therefore we get that  the following inequality holds
\begin{equation}\label{c50}
\int u Zu (r^2 - \rho^2)^{\alpha}\psi  +  8 \sqrt{K} r^2 H(r)  + f(r) H(r) \geq 7 \sqrt{K} r^2 H(r)
\end{equation}

Now because we are in \textbf{Case} 2, \eqref{c2} and \eqref{c50} together imply that 
\begin{equation}
(\int u^2 (r^2 - \rho^2)^{\alpha} \psi)^{1/2}(\int (Zu)^2 (r^2 - \rho^2)^{\alpha} \psi)^{1/2} \geq  7 \sqrt{2}    \sqrt{K} r^2  H(r)
\end{equation}
By squaring the above inequality and by cancelling off $H(r)$ from  both sides, we get that
\begin{equation}\label{pin1}
(\int (Zu)^2 (r^2 - \rho^2)^{\alpha} \psi ) \geq 94 K r^4 H(r)
\end{equation}

 By dividing both sides by $rH(r)$ we get

 \begin{equation}\label{p4}
 \frac{(\alpha+1) \int (Zu)^2 (r^2-\rho^2)^{\alpha} \psi}{ r H(r)} \geq 94  K^{3/2} r^3
 \end{equation}
  Therefore by writing 
\[  
   \frac{2(\alpha+1) \int (Zu)^2 (r^2-\rho^2)^{\alpha} \psi}{ r H(r)} = \frac{(\alpha+1) \int (Zu)^2 (r^2-\rho^2)^{\alpha} \psi}{ r H(r)} +\frac{(\alpha+1) \int (Zu)^2 (r^2-\rho^2)^{\alpha} \psi}{ r H(r)} 
  \]
  and by using \eqref{p4} in \eqref{pin4}, we get that
 \begin{align}\label{pin5'}
 & N'(r) \geq   -C_7 K r -  C_{8} K \frac{f(r)}{r}  + 94 K^{3/2} r^3
 \\
 & + \frac{(\alpha+1) \int (Zu)^2 (r^2-\rho^2)^{\alpha} \psi}{ r H(r)}  
 \notag
 \\
 &- 16 K^{3/2} r^3  - \frac{8(\alpha+1)f(r) (\int |u| |Zu| (r^2- \rho^2)^{\alpha} \psi)}{r H(r)}
 \notag
 \end{align}
 Therefore we see that  the  estimate \eqref{p4} allows us to get rid of the undesirable term $16 K^{3/2} r^3$ and  now from \eqref{pin5'}  one can easily infer that the following inequality  holds
 \begin{align}\label{pin5}
 & N'(r) \geq   -C_7 K r -  C_{8} K \frac{f(r)}{r}  
 \\
 & + \frac{(\alpha+1) \int (Zu)^2 (r^2-\rho^2)^{\alpha} \psi}{ r H(r)}  
 \notag
 \\
 & - \frac{8(\alpha+1)f(r) (\int |u| |Zu| (r^2- \rho^2)^{\alpha} \psi)}{r H(r)}
 \notag
 \end{align} 
  Again by applying Cauchy-Schwartz inequality with $\ve$ to the term 
  
  \[
   \frac{8(\alpha+1)f(r) (\int |u| |Zu| (r^2- \rho^2)^{\alpha} \psi)}{r H(r)}
   \] 
  with an appropriate choice of $\ve$, we  get  that  the following estimate  holds 
  \begin{equation}\label{jn5}
  N'(r) \geq - C_7 K r - C_9 K \frac{f(r)}{r}
  \end{equation}
  Therefore  in conclusion,  we have that  in all the cases,  either the estimate  \eqref{ct1}, \eqref{calc7} or \eqref{jn5}  holds.   Each  of these estimates implies that for some universal constant $\tilde{C_0}$,   the following inequality  holds
  \begin{equation}\label{jn7}
  N'(r) \geq -\tilde{C_0}( \frac{f(r)}{r} N(r) + Kr + K \frac{f(r)}{r})
  \end{equation}
  \eqref{m1} now follows from \eqref{jn7} in a standard way.

\end{proof}

\subsection{Proof of  estimate \eqref{est} in Theorem \ref{main}}
We note that although \eqref{m} in the monotonicity Theorem \ref{mon} is different from its counterpart Theorem 3.1 in \cite{BG1}, nevertheless it still implies that the following inequality holds
\begin{equation}\label{ineq}
N(r) \leq  \tilde{C_1} (N(s) + \tilde{C_2} K),\ \ \ \ \ \ \ \ \ \text{for}\  0<r<s< 1.
\end{equation}
Using \eqref{ineq}, we can argue in the same way as  in Section 4 in \cite{BG1} to conclude that our desired estimate \eqref{est} in Theorem \ref{main} holds.

\end{document}